\documentclass[a4paper,12pt]{amsart}
\usepackage[margin=2.5cm]{geometry}
\usepackage[utf8]{inputenc}
  \usepackage[T1]{fontenc}
  \usepackage[charter]{mathdesign}
  \usepackage{microtype}
\usepackage{hyperref}
\hypersetup{
bookmarks=true,
pdfpagemode=UseNone,
pdfstartview=FitH,
pdfdisplaydoctitle=true,
pdftitle={Variation estimates for averages along primes and polynomials},
pdfauthor={Pavel Zorin-Kranich},
pdfsubject={Mathematics Subject Classification (2010): 42A45 (Primary) 37A45, 26A45 (Secondary)},
pdfkeywords={variation norm, ergodic averages, prime numbers},
pdflang=en-US
}

\usepackage[safeinputenc=true,style=alphabetic,firstinits,useprefix=true,maxnames=4,%
doi=false,url=false,isbn=false,eprint=true,backref=false,backend=bibtex%
]{biblatex}
\newbibmacro{string+urldoi}[1]{%
  \iffieldundef{url}{%
    \iffieldundef{doi}{%
        #1%
      }{
        \href{http://dx.doi.org/\thefield{doi}}{#1}%
    }%
  }{
    \href{\thefield{url}}{#1}%
  }%
}
\DeclareFieldFormat[article,misc,incollection,book]{title}{\usebibmacro{string+urldoi}{\mkbibquote{#1}}}

\numberwithin{equation}{section}
\newtheorem{theorem}[equation]{Theorem}
\newtheorem{lemma}[equation]{Lemma}
\newtheorem{proposition}[equation]{Proposition}
\newtheorem{corollary}[equation]{Corollary}
\theoremstyle{definition}
\newtheorem{definition}[equation]{Definition}
\theoremstyle{remark}

\newcommand{\R}{\mathbb{R}}
\newcommand{\gJ}{\mathscr{N}} 
\newcommand{\lJ}{\mathfrak{N}} 
\newcommand{\hV}{\mathscr{\tilde V}} 
\newcommand{\iV}{\mathscr{V}} 
\newcommand{\sV}{\mathrm{s}\mathscr{\tilde V}} 
\DeclareMathOperator{\lcm}{lcm}
\newcommand{\Ks}[1][N]{K_{#1}}
\newcommand{\Kf}[1][N]{\widehat{\Ks[#1]}}
\newcommand{\Ls}[1][N]{L_{#1}}
\newcommand{\LsR}[1][N]{L^{\R}_{#1}}
\newcommand{\Lf}[1][N]{\widehat{\Ls[#1]}}
\newcommand{\LfR}[1][N]{\widehat{\LsR[#1]}}
\newcommand{\cm}[1][t]{\sigma_{#1}} 
\newcommand{\cmf}[1][t]{\widehat{\cm[#1]}} 
\newcommand{\F}{\mathscr{F}}
\renewcommand{\Finv}{\F^{-1}}
\newcommand{\cutoff}[1][]{\eta_{#1}}
\newcommand{\cutofff}[1][]{\widehat{\cutoff[#1]}}
\newcommand{\rats}[1][s]{\mathcal{R}_{#1}} 
\newcommand{\Nker}{\mathfrak{I}} 
\newcommand{\Smax}[1][s]{\mathbf{S}_{#1}}
\newcommand{\C}{\mathbb{C}}
\newcommand{\N}{\mathbb{N}}
\newcommand{\Z}{\mathbb{Z}}
\newcommand{\Q}{\mathbb{Q}}

\newcommand{\dif}{\mathrm{d}}
\DeclareMathOperator{\supp}{supp}
\newcommand{\tot}{\varphi} 
\def\<{\left\langle}
\def\>{\right\rangle}

\begin{document}
\subjclass[2010]{42A45 (Primary) 37A45, 26A45 (Secondary)}
\title{Variation estimates for averages along primes and polynomials}
\author{Pavel Zorin-Kranich}
\address{Institute of Mathematics\\
Hebrew University, Givat Ram\\
Jerusalem, 91904, Israel}
\thanks{Research supported by the ISF grant 1409/11.}
\keywords{variation norm, ergodic averages, prime numbers}
\begin{abstract}
We prove $q$-variation estimates, $q>2$, on $\ell^{p}$ spaces for averages along primes (with $1<p<\infty$) and polynomials (with $| 1/p - 1/2 | < 1/2(d+1)$, where $d$ is the degree of the polynomial).
This improves the pointwise ergodic theorems for these averages in the corresponding ranges of $L^{p}$ spaces.
\end{abstract}
\maketitle

\section{Introduction}
Variation and oscillation estimates for convolution operators associated to polynomials and primes have been pioneered by Bourgain in order to prove the corresponding pointwise ergodic theorems \cite{MR937582,MR950982,MR1019960}.
For the ordinary Ces\`aro averages the full range of expected variation estimates has been obtained by Jones, Kaufman, Rosenblatt, and Wierdl \cite{MR1645330} and for averages along scalar polynomials variation estimates on $\ell^{2}(\Z)$ have been obtained by Krause \cite{2014arXiv1402.1803K}.
Here we prove the corresponding estimates for averages along primes on $\ell^{p}(\Z)$ and along vector-valued polynomials on $\ell^{p}(\Z^{d})$ for $p$ in certain open ranges.
See \textsection\ref{sec:var} for the definition of the variation norms $\iV^{q}$ and the relation $\lesssim$.
\begin{theorem}
\label{thm:var-primes}
Let
\begin{equation}
\label{eq:KN:primes}
\Ks = \frac1N \sum_{n\leq N} \Lambda(n) \delta_{n},
\end{equation}
where $\Lambda$ denotes the von Mangoldt function.
Then for any $1 < p < \infty$ and $q>2$ we have
\[
\| \| \Ks * f \|_{\iV^{q}_{N\in\N}} \|_{\ell^{p}(\Z)} \lesssim_{p,q} \|f\|_{\ell^{p}(\Z)}.
\]
\end{theorem}

\begin{theorem}
\label{thm:var-poly}
Let $d\geq 1$ and
\begin{equation}
\label{eq:KN:poly}
\Ks = \frac1N \sum_{n\leq N} \delta_{(n^{1},n^{2},\dots,n^{d})}.
\end{equation}
Then for any $p$ with $\big| \frac1p - \frac12 \big| < \frac{1}{2(d+1)}$ and any $q>2$ we have
\[
\| \| \Ks * f \|_{\iV^{q}_{N\in\N}} \|_{\ell^{p}(\Z^{d})} \lesssim_{p,q} \|f\|_{\ell^{p}(\Z^{d})}.
\]
\end{theorem}
Theorem~\ref{thm:var-poly} is not likely to be optimal as far as the restriction on $p$ is concerned, and in fact we believe that it should extend to $1<p<\infty$.
Interpolation shows that the set of pairs $(1/p,1/q)$ for which this result holds is convex, see e.g.\ \cite[\textsection 7]{2014arXiv1402.1803K}.
In view of the maximal inequality (see \cite[Theorem D]{2014arXiv1405.5566M} or \cite[(7.1)]{MR1019960} for the linearly dependent case), corresponding to $1<p<\infty$ and $q=\infty$, this yields partial results (with a smaller range of $q$'s) towards extending the range of allowed $p$'s.

The proofs of Theorems~\ref{thm:var-primes} and \ref{thm:var-poly} follow the lines of Bourgain's article \cite{MR1019960} but use the more recent variational estimates for convolutions \cite{MR2434308} and trigonometric polynomials \cite{MR2653686} in order to obtain an appropriate multi-frequency variational inequality on $L^{2}(\R^{d})$.
Multi-frequency $L^{p}$ estimates are obtained in two different ways.
The first way consists in interpolation between multi-frequency $L^{2}$ estimates and single-frequency $L^{p}$ estimates.
This approach does not rely on algebraic relations between the distinguished frequencies, but does not yield optimal estimates in our cases.
The second way is more specific to our algebraic setting and goes back at least to Wierdl \cite{MR995574}, although the lack of an easy endpoint at $p=\infty$ (as for the maximal inequality) necessitates the use of a more recent transfer technique from \cite{MR1888798}.
The multi-frequency variational estimates are applied to certain Fourier multipliers that approximate $\Kf$ in a sufficiently strong sense.
The construction of these multipliers is due to Bourgain \cites{MR937581,MR950982}.
We include concise proofs of their properties established in \cite{MR950982} and \cite{MR995574}.
A technical novelty of our argument is that the multi-frequency estimates are used to control the variation norm on a sequence of times which is denser than dyadic.
This simplifies the treatment of short variations, see Lemma~\ref{lem:short-var-lp}.

By Calder\'on's transference principle \cite[Theorem 1]{MR0227354} Theorems \ref{thm:var-primes} and \ref{thm:var-poly} imply the following analogous statements for measure-preserving $\Z$-actions on $\sigma$-finite measure spaces.
\begin{corollary}
Let $(X,\mu,T)$ be a ($\sigma$-finite, invertible) measure-preserving system.
Then for every $1<p<\infty$ and $q>2$ we have
\[
\| \| \frac1N \sum_{n=1}^{N} \Lambda(n) f(T^{n}x) \|_{\iV^{q}_{N\in\N}} \|_{L^{p}_{x}} \lesssim_{p,q} \| f \|_{L^{p}}
\quad\text{for every } f\in L^{p}(X,\mu).
\]
\end{corollary}
\begin{corollary}
\label{cor:var-poly}
Let $(X,\mu)$ be a $\sigma$-finite measure space and $T_{1},\dots,T_{d}:X\to X$ be commuting invertible measure-preserving transformations.
Then for every $p$ with $\big| \frac1p - \frac12 \big| < \frac{1}{2(d+1)}$ and every $q>2$ we have
\[
\| \| \frac1N \sum_{n=1}^{N} f(T_{1}^{n^{1}}\cdots T_{d}^{n^{d}} x) \|_{\iV^{q}_{N\in\N}} \|_{L^{p}_{x}} \lesssim_{p,q} \| f \|_{L^{p}(X)}
\quad\text{for every } f\in L^{p}(X,\mu).
\]
\end{corollary}
Corollary~\ref{cor:var-poly} applies in particular in the case $X=\Z$, $T_{j}=T^{a_{j}}$, $T:\Z\to\Z$ the shift.
We find it convenient to exclude this non-homogeneous situation from Theorem~\ref{thm:var-poly} because the homogeneous setup offers a more direct link to the dilation-invariant results from \cite{MR2434308}.

I thank Mariusz Mirek for pointing out an error in an earlier revision of this text.

\section{Variation of exponential sums}
A pointwise variational estimate for an exponential sum with fixed, separated frequencies and varying coefficients was a central technical innovation in the article of Nazarov, Oberlin, and Thiele \cite[Lemma 3.2]{MR2653686}.
We will need a version of this result for exponential sums on $\R^{d}$, $d\geq 1$, with explicit dependence of the constants on all parameters.
The higher-dimensional version follows from a more abstract formulation due to Oberlin \cite[Proposition 9.3]{MR3090139}, explicit constants have been given by Krause \cite[Lemma 2.4]{2014arXiv1402.1804K}, and it is clear how these extensions should be combined.
We include a detailed proof because our construction of the ``parent'' function $\rho$ is slightly simpler than that used in the articles cited above.
We begin with a short summary of the relevant definitions.
\label{sec:var}
\begin{definition}
Let $\Nker$ be a totally ordered set and $(c_{t})_{t\in\Nker}$ be an $\Nker$-sequence in a normed space.
We denote
\begin{enumerate}
\item by $\gJ_{\lambda}(c)$, $\lambda>0$, the \emph{greedy jump counting function}, that is, the supremum over the lengths $J$ of sequences $t_{0} < t_{1} < \dots < t_{J}$ such that $|c_{t_{j}}-c_{t_{j-1}}|>\lambda$ for all $j=1,\dots,J$,
\item by $\lJ_{\lambda}(c)$, $\lambda>0$, the \emph{lazy jump counting function}, that is, the supremum over the lengths $J$ of sequences $s_{1}<t_{1} \leq s_{2}<t_{2} \leq \dots \leq s_{J}<t_{J}$ such that $|c_{t_{j}}-c_{s_{j}}|>\lambda$ for all $j=1,\dots,J$,
\item by $\hV^{q}(c)=\|c_{t}\|_{\hV^{q}_{t}}$, $q>0$, the \emph{homogeneous $q$-variation norm}, that is, the supremum of
\[
\| c_{t_{j+1}}-c_{t_{j}} \|_{\ell^{q}_{j}}
\]
over all strictly increasing sequences $t_{1}<\dots<t_{J}$, and
\item the \emph{inhomogeneous $q$-variation norm} by
\[
\iV^{q}(c)=\|c_{t}\|_{\iV^{q}_{t}} = ((\hV^{q})^{q}+(\sup_{t} |c_{t}|)^{q})^{1/q}.
\]
\end{enumerate}
\end{definition}
We will sometimes write $\gJ_{\lambda,t}$, $\iV^{q}_{t}$, $\iV^{q}_{t\in\Nker}$, etc., in order to emphasize the relevant variable and $\hV^{q}(X)$ in order to emphasize the normed space in which the sequence $(c_{t})$ takes values.
It is clear that both $\gJ_{\lambda}$ and $\lJ_{\lambda}$ are monotonically decreasing in $\lambda$ and
\[
\gJ_{\lambda} \leq \lJ_{\lambda} \leq \gJ_{\lambda/2}.
\]
Moreover, we can pass between variation and jump estimates using the identities
\begin{equation}
\label{eq:jump<var}
\lambda \gJ_{\lambda}^{1/q} \leq \lambda \lJ_{\lambda}^{1/q} \leq \hV^{q}
\end{equation}
and
\begin{equation}
\label{eq:var<jump}
\hV^{q} \leq (\sum_{k\in\Z} (2^{k+1})^{q} \lJ_{2^{k}})^{1/q} \leq 4 (\sum_{k\in\Z} (2^{k})^{q} \gJ_{2^{k}})^{1/q}.
\end{equation}
Note that the inhomogeneous variation norm is controlled by the homogeneous variation norm and the value of the sequence at any given point $t$.
Estimates at a fixed $t$ will be easy in many of our variation inequalities, allowing us to concentrate on the homogeneous variation norm.

A recurring theme will be splitting the variation into a ``long'' and a ``short'' part with respect to an increasing, cofinal, and coinitial sequence $Z=\{\dots,N_{1},N_{2},\dots\}$ in $\Nker$.
The \emph{long variation} of a sequence $(c_{t})$ with respect to $Z$ is simply $\| c_{t} \|_{\hV^{q}_{t\in Z}}$.
The \emph{short variation} with respect to $Z$ is defined by
\[
\| c_{t} \|_{\sV^{q}} = \Big( \sum_{j} \|c_{t}\|_{\hV^{q}_{t\in [N_{j},N_{j+1}]}}^{q} \Big)^{1/q}.
\]
It is well-known that the full homogeneous variation is controlled by the long and the short variation, namely
\begin{equation}
\label{eq:var-long-short}
\| c_{t} \|_{\hV^{q}_{t}} \leq \| c_{t} \|_{\hV^{q}_{t\in Z}} + 2\| c_{t} \|_{\sV^{q}}.
\end{equation}
To see this, consider any sequence $t_{1}<\dots<t_{J}$ as in the definition of the homogeneous variation norm.
For every $j$ we have $t_{j}\in [N_{j_{-}},N_{j_{+}}]$ with $j_{+}=j_{-}+1$.
If $t_{j}\leq N_{j_{+}} \leq N_{(j+1)_{-}} \leq t_{j+1}$, then we split the corresponding difference $c_{t_{j}}-c_{t_{j+1}}$ accordingly, otherwise we have $t_{j+1}\in [N_{j_{-}},N_{j_{+}}]$.
Thus the sequence
\[
(c_{t_{j}}-c_{t_{j+1}})_{j}
\]
can be written as the sum of three sequences, one of which corresponds to differences between $N_{j}$'s and the others to differences within intervals $[N_{j},N_{j+1}]$.
Taking the supremum over all increasing sequences of $t_{j}$'s we obtain the claim.

\begin{lemma}
\label{lem:alm-orth}
Let $I=I_{1}\times\dots\times I_{d}\subset \R^{d}$ be a product of intervals and let $(\xi_{\vec k})_{\vec k \in \Z^{d}} \subset \R^{d}$ be frequencies such that $(\xi_{\vec k,i}-\xi_{\vec l,i}) \gtrsim |k_{i}-l_{i}|/|I_{i}|$.
Then we have
\[
\| \sum_{k\in\Z^{d}} c_{k} e(\xi_{k}\cdot y) \|_{L^{2}_{y}(I)}
\lesssim |I|^{1/2} \|c_{k}\|_{\ell^{2}_{k}},
\]
where the implied constant depends only on the implied constant in the hypothesis and the dimension $d$.
\end{lemma}
Here and later $C_{a}$ denotes an unspecified positive constant, depending on auxiliary parameter(s) $a$, whose value may vary from line to line.
We say that $A$ is dominated by $B$, in symbols $A\lesssim_{a} B$, if $A\leq C_{a}B$.
The parameters $a$ can be partially or fully omitted if they are clear from the context.
\begin{proof}
Let $w_{i}$ be a smooth non-negative functions bounded by $1$ and supported on $2I_{i}$ with $w_{i}|_{I_{i}} \equiv 1$ and $|w_{i}''| |I_{i}|^{2} \lesssim 1$.
Let also $w(y)=w_{1}(y_{1})\cdots w_{d}(y_{d})$.
We use almost-orthogonality of the phases $e(\xi_{k} \cdot y)$ in $L^{2}(w)$.
More precisely, by partial integration we obtain
\begin{multline*}
\| \sum_{k} c_{k} e(\xi_{k}\cdot y) \|_{L^{2}_{y}(I)}^{2}
\leq
\| \sum_{k} c_{k} e(\xi_{k}\cdot y) \|_{L^{2}_{y}(w)}^{2}
\leq
\sum_{k,l\in\Z^{d}} \Big| c_{k}\overline{c_{l}} \int e((\xi_{k}-\xi_{l}) \cdot y) w(y) \dif y \Big|\\
\leq
\sum_{k,l\in\Z^{d}} |c_{k}c_{l}| \Big| \prod_{i:k_{i}\neq l_{i}} ((2\pi i) (\xi_{k,i}-\xi_{l.i}))^{-2} \int e((\xi_{j}-\xi_{k}) \cdot y) (\prod_{i:k_{i}\neq l_{i}} \partial_{i}^{2}) w(y) \dif y \Big|\\
\lesssim
\sum_{k,l\in\Z^{d}} |c_{k}c_{l}| \big(\prod_{i:k_{i}\neq l_{i}} |I_{i}|^{2} (k_{i}-l_{i})^{-2} \big) |I| \big(\prod_{i:k_{i}\neq l_{i}} |I_{i}|^{-2} \big)\\
\leq
|I| \sum_{k,l\in\Z^{d}} (|c_{k}|^{2}+|c_{l}|^{2})/2 \prod_{i:k_{i}\neq l_{i}} (k_{i}-l_{i})^{-2}
\end{multline*}
Since the last expression is symmetric in $k$ and $l$, it is bounded by
\begin{multline*}
|I| \sum_{k,l\in\Z^{d}} |c_{k}|^{2} \prod_{i:k_{i}\neq l_{i}} (k_{i}-l_{i})^{-2}
=
|I| \sum_{k\in\Z^{d}} |c_{k}|^{2} \sum_{l\in\Z^{d}} \prod_{i:k_{i}\neq l_{i}} (k_{i}-l_{i})^{-2}
=
|I| \sum_{k\in\Z^{d}} |c_{k}|^{2} \sum_{l\in\Z^{d}} \prod_{i:l_{i} \neq 0} l_{i}^{-2}.
\end{multline*}
Since the last sum over $l$ is finite, we obtain the claim.
\end{proof}

The next lemma captures the main step in the proof of \cite[Lemma 3.2]{MR2653686}.
In the formulation below the left-hand side is essentially from \cite[Proposition 9.3]{MR3090139} and the right-hand side is essentially from \cite[Lemma 2.4]{2014arXiv1402.1804K}.
\begin{lemma}
\label{lem:NOT:3.2}
Let $B$ be a normed space, $I$ a measure space, and let $g\in L^{r}(I,B')$, $r\geq 1$.
Let also $(c_{t})_{t\in\Nker} \subset B$ with a countable totally ordered set $\Nker$, and $q>r$.
Then
\[
\| \| \< c_{t}, g(y)\> \|_{\hV^{q}_{t}} \|_{L^{r}_{y}(I)}
\lesssim
\int_{0}^{\infty} \min( M \gJ_{\lambda}^{1/r}, \|g\|_{L^{r}(I,B')} \gJ_{\lambda}^{1/q}) \dif\lambda,
\]
where $\gJ$ is the greedy jump counting function for the sequence $(c_{t})$,
\[
M := \sup_{c\in B, \|c\|=1} \| \<c, g(y)\> \|_{L^{r}_{y}(I)},
\]
and the implied constant is absolute.
\end{lemma}

\begin{proof}
It suffices to consider finite sequences $(c_{t})_{t=1}^{T}$ as long as the bounds do not depend on $T$.
We may assume that the minimal jump size $\min_{t<T} \|c_{t}-c_{t+1}\|_{B} > \lambda > 0$, otherwise one can remove some of the terms from the sequence $(c_{t})$.
We construct a sequence of increasingly coarse partitions of $\{1,\dots,T\}$ into blocks with bounded $\infty$-variation and jumps between blocks in such a way that both the upper bounds on the $\infty$-variation and the lower bounds on the jumps increase exponentially.
To this end we recursively define a sequence of functions $\rho(n,\cdot) : \{1,\dots,T\} \to \{1,\dots,T\}$.
We begin with
\[
\rho(0,t) = t.
\]
Suppose that $\rho(n,\cdot)$ has been defined for some $n$ and define $\rho(n+1,t)$ by recursion in $t$ starting with $\rho(n+1,1)=1$ by
\[
\rho(n+1,t+1) :=
\begin{cases}
\rho(n+1,t) & \text{if } \|c_{\rho(n+1,t)} - c_{\rho(n,t+1)}\|_{B} \leq 2^{n+1}\lambda,\\
\rho(n,t+1) & \text{otherwise.}
\end{cases}
\]
It follows that $\rho(n,t)$ is monotonically increasing in $t$ and monotonically decreasing in $n$.
Moreover, for all $n$ and $t$ we have
\begin{align}
&\|c_{\rho(n,t)} - c_{\rho(n,t+1)}\|_{B} > 2^{n}\lambda  \qquad\text{provided } \rho(n,t)\neq \rho(n,t+1),\label{eq:lem:NOT:3.2:jump-size-lower}\\
&\|c_{\rho(n,t)}-c_{\rho(n+1,t)}\|_{B} \leq 2^{n+1} \lambda, \label{eq:lem:NOT:3.2:jump-size-upper}\\
&\rho(n+1,t+1) \neq \rho(n+1,t) \implies \rho(n,t+1)\neq\rho(n,t). \label{eq:lem:NOT:3.2:jump-n-k}
\end{align}
The implication \eqref{eq:lem:NOT:3.2:jump-n-k} can be easily seen by the contrapositive and a case distinction in the definition of $\rho(n+1,t)$.
Note that \eqref{eq:lem:NOT:3.2:jump-size-lower} implies $\rho(n,t)=1$ for all $t$ if $n$ is sufficiently large.
Write
\[
c_{t} = c_{1} + \sum_{n=0}^{\infty} (c_{\rho(n,t)}-c_{\rho(n+1,t)}).
\]
By subadditivity of the homogeneous variation norm we have
\[
\| \| \< c_{t}, g(y) \> \|_{\hV^{q}_{t}} \|_{L^{r}_{y}(I)}
\leq
\sum_{n=0}^{\infty} \| \| \<c_{\rho(n,t)}-c_{\rho(n+1,t)}, g(y)\> \|_{\hV^{q}_{t}} \|_{L^{r}_{y}(I)}.
\]
For each $n$ we estimate the corresponding summand.
Observe that the lower bound on the jump size in \eqref{eq:lem:NOT:3.2:jump-size-lower} implies that the sequence $\rho(n,\cdot)$ makes at most $\gJ_{2^{n}\lambda}$ jumps, before places $J_{n}\subset\{1,\dots,T\}$, say.
Note that $J_{n+1} \subset J_{n}$ by \eqref{eq:lem:NOT:3.2:jump-n-k}.
Hence the variation norm in the summand collapses to the subsequence
\[
J_{n}' =
\begin{cases}
\{1\} \cup J_{n} & \text{if } J_{n}\neq\emptyset,\\
\emptyset & \text{otherwise.}
\end{cases}
\]
On that subsequence we estimate the $\hV^{q}$ norm by the $\ell^{q}$ norm, thereby obtaining the following bound for the $n$-th summand:
\begin{equation}
\label{eq:lem:NOT:3.2:summand}
\| \| \<c_{\rho(n,t)}-c_{\rho(n+1,t)}, g(y)\> \|_{\ell^{q}_{t\in J_{n}'}} \|_{L^{r}_{y}(I)}.
\end{equation}
The first way to proceed from here is to estimate the $\ell^{q}$ norm by the $\ell^{r}$ norm and to change the order of integration (in $t$ and $y$).
Using \eqref{eq:lem:NOT:3.2:jump-size-upper} this gives the bound
\[
M 2^{n+1} \lambda |J_{n}'|^{1/r}
\lesssim
M 2^{n} \lambda \gJ_{2^{n}\lambda}^{1/r}.
\]
The second way to proceed is to estimate the dual pairing by the product of norms, which gives for \eqref{eq:lem:NOT:3.2:summand} the estimate
\[
\| \| \|c_{\rho(n,t)}-c_{\rho(n+1,t)}\|_{B} \|g(y)\|_{B'} \|_{\ell^{q}_{t\in J_{n}'}} \|_{L^{r}_{y}(I)}
\]
By \eqref{eq:lem:NOT:3.2:jump-size-upper} this gives the bound
\[
\|g\|_{L^{r}(I,B')} 2^{n+1}\lambda |J_{n}'|^{1/q}
\lesssim
\|g\|_{L^{r}(I,B')} 2^{n}\lambda \gJ_{2^{n}\lambda}^{1/q}.
\]
Combining these estimates we obtain
\[
\| \| \< c_{t}, g(y)\> \|_{\hV^{q}_{t}} \|_{L^{r}_{y}(I)}
\lesssim
\sum_{n\in\N} 2^{n}\lambda \min(M \gJ_{2^{n}\lambda}^{1/r}, \|g\|_{L^{r}(I,B')} \gJ_{2^{n}\lambda}^{1/q}),
\]
and the claim follows by monotonicity of the jump counting function.
\end{proof}

\begin{corollary}[{cf.\ \cite[Lemma 3.2]{MR2653686}}]
\label{cor:NOT:3.2}
Let $G\subset \Z^{d}$ be a set of size $N$ and $(\xi_{l})_{l\in G} \subset \R^{d}$, $I_{i}$, $I$ be as in Lemma~\ref{lem:alm-orth}.
For any $2<r<q$ and any sequence $(c_{t}) \subset \ell^{2}[G]$ we have
\[
\| \| \sum_{l\in G} c_{t,l} e(\xi_{l}\cdot y) \|_{\hV^{q}_{t}} \|_{L^{2}_{y}(I)}
\lesssim
\left( \frac{q}{q-r} + \frac{2}{r-2} \right)
|I|^{1/2} N^{(\frac12 - \frac1r) \frac{q}{q-2}} \|c_{t}\|_{\hV^{r}_{t}(\ell^{2}[G])},
\]
where the implied constant is absolute.
\end{corollary}
\begin{proof}
We apply Lemma~\ref{lem:NOT:3.2} with $B=B'=\ell^{2}[G]$, $r=2$, and $g(y)=(e(\xi_{l}\cdot y))_{l\in G}$.
Then $\|g\|_{L^{2}(I,B')} = N^{1/2} |I|^{1/2}$ and $M\lesssim |I|^{1/2}$ by Lemma~\ref{lem:alm-orth}.
Hence we obtain
\[
\| \| \sum_{l\in G} c_{t,l} e(\xi_{l}\cdot y) \|_{\hV^{q}_{t}} \|_{L^{2}_{y}(I)}
\lesssim
|I|^{1/2} \int_{0}^{\infty} \min(\gJ_{\lambda}^{1/2}, N^{1/2} \gJ_{\lambda}^{1/q}) \dif\lambda.
\]
We have $\gJ_{\lambda}(c) \leq a^{r}/\lambda^{r}$ with $a=\|c_{t}\|_{\hV^{r}_{t}(\ell^{2}[G])}$.
Splitting the integral at $\lambda_{0}=aN^{-1/(2r(1/2-1/q))}$ we obtain
\begin{multline*}
\int_{0}^{\lambda_{0}} N^{1/2} (a^{r}/\lambda^{r})^{1/q} \dif\lambda
+
\int_{\lambda_{0}}^{\infty} (a^{r}/\lambda^{r})^{1/2} \dif\lambda\\
=
N^{1/2}a^{r/q} (-r/q+1)^{-1} \lambda_{0}^{-r/q+1}
-
a^{r/2} (-r/2+1)^{-1} \lambda_{0}^{-r/2+1}\\
=
a N^{(\frac12 - \frac1r)\frac{q}{q-2}} ((1-r/q)^{-1} + (r/2-1)^{-1}).
\qedhere
\end{multline*}
\end{proof}

\section{Fourier multipliers on $\R^{d}$}
\label{sec:R}
The main result of this section, Proposition~\ref{prop:4.13}, is a multiple-frequency variation inequality on $L^{2}(\R^{d})$ with a good (logarithmic) dependence of the bounds on the number of frequencies involved in it.
We begin by recalling several variation inequalities due to Jones, Seeger, and Wright, limiting ourselves to the minimal level of generality required in our applications.
The first is a special case of \cite[Lemma 2.1]{MR2434308}, which goes back to Bourgain's argument from \cite[\textsection 3]{MR1019960}.
\begin{lemma}
\label{lem:lambdaN-to-Vq-bound}
Let $(X,\mu)$ be a measure space and $(T_{i})_{i\in \Nker\subset\R}$ be a family of continuous linear operators on $L^{p}(X)$, $1<p<\infty$, that are contractive on $L^{\infty}(X)$ and such that $T_{i}f(x)$ is continuous in $i$ for almost every $x$.
Suppose that
\begin{equation}
\label{eq:lambdaN-to-Vq-bound:assumption}
\sup_{\lambda>0} \| \lambda (\gJ_{\lambda,i}(T_{i}f)(x))^{1/2} \|_{L^{r}(X,\mu)}
\lesssim_{r} \| f \|_{L^{r}(X,\mu)}
\end{equation}
for every $1<r<\infty$ and every characteristic function $f=\chi_{A}$ of a finite measure subset $A\subset X$.
Then for every $q>2$ and $1<p<\infty$ we have
\begin{equation}
\label{eq:lambdaN-to-Vq-bound:conclusion}
\| \|T_{i}f(x)\|_{\hV^{q}_{i\in\Nker}} \|_{L^{p}(X,\mu)} \lesssim_{p} \frac{q}{q-2} \| f \|_{L^{p}(X,\mu)}.
\end{equation}
\end{lemma}
Let us point out how the various qualitative assumptions are used in the proof of \cite[Lemma 2.1]{MR2434308}.
By the qualitative assumption of pointwise continuity almost everywhere the problem reduces to countable index sets $I\subset\R$, and in particular the jump counting functions and the pointwise variation norms become measurable functions on $X$.
This in turn allows one to use monotone convergence to reduce the problem to finite sets $I$.
The proof proceeds by establishing restricted strong type estimates, which are then interpolated to the requested strong type estimates.
However, these are a priori obtained for simple functions (finite linear combinations of characteristic functions), and the qualitative assumption of $L^{p}$ continuity of the individual operators is needed to pass to the full $L^{p}$ space.

\subsection{A variation inequality for a single frequency}
We will apply Lemma~\ref{lem:lambdaN-to-Vq-bound} in the setting of convolution operators.
Let $\cm$ be the measure on $\R^{d}$ defined by
\begin{equation}
\label{eq:cm}
\int f \dif\cm = \frac1t \int_{s=0}^{t} f(s^{1},s^{2},\dots,s^{d}) \dif s.
\end{equation}
The following result is stated in a remark following \cite[Theorem 1.5]{MR2434308}.
\begin{theorem}
\label{thm:conv-var}
For any $1<p<\infty$ we have
\[
\sup_{\lambda>0} \| \lambda \lJ_{\lambda,t\in\R}( f * \cm(x) )^{1/2} \|_{L^{p}_{x}(\R^{d})}
\lesssim_{p} \|f\|_{L^{p}(\R^{d})}.
\]
\end{theorem}
\begin{corollary}
\label{cor:3.30a}
For any $1<p<\infty$ and any $s>2$ we have
\[
\| \| f * \cm(x) \|_{\hV^{s}_{t>0}} \|_{L^{p}_{x}(\R^{d})}
\lesssim_{p}
\frac{s}{s-2} \|f\|_{L^{p}(\R^{d})},
\]
where the implied constant does not depend on $s$.
\end{corollary}
\begin{proof}
One can show that the function $s\mapsto f(x_{1}+s^{1},\dots,x_{d}+s^{d})$ is locally integrable for almost every $(x_{1},\dots,x_{d})$ using Fubini's theorem.
This implies that $t\mapsto f * \cm(x)$ is continuous in $t$, and in view of Theorem~\ref{thm:conv-var} we may apply Lemma~\ref{lem:lambdaN-to-Vq-bound}.
\end{proof}

This implies the following variation version of \cite[Lemma 3.30]{MR1019960}.
\begin{corollary}
\label{cor:3.30}
For any $2\leq p<\infty$, any set $G$, and any $s>2$ we have
\[
\| \| f_{l} * \cm(x) \|_{\hV^{s}_{t>0}(\ell^{2}_{l\in G})} \|_{L^{p}_{x}(\R^{d})}
\lesssim_{p}
\frac{s}{s-2} \| \|f_{l}\|_{L^{p}(\R^{d})} \|_{\ell^{2}_{l\in G}},
\]
where the implied constant does not depend on $G$ and $s$.
\end{corollary}
\begin{proof}
The case $|G|=1$ is given by Corollary~\ref{cor:3.30a}.

In the general case for any finite sequence $t_{1}<\dots<t_{J}$ and $x\in\R$ we have
\begin{multline*}
\| f_{l}*\cm[t_{j}](x) - f_{l}*\cm[t_{j-1}](x) \|_{\ell^{s}_{j}(\ell^{2}_{l})}
=
\| \| f_{l}*\cm[t_{j}](x) - f_{l}*\cm[t_{j-1}](x) \|_{\ell^{2}_{l}} \|_{\ell^{s}_{j}}\\
\leq
\| \| f_{l}*\cm[t_{j}](x) - f_{l}*\cm[t_{j-1}](x) \|_{\ell^{s}_{j}} \|_{\ell^{2}_{l}}
\leq
\| \| f_{l}*\cm(x) \|_{\hV^{s}_{t}} \|_{\ell^{2}_{l}}
\end{multline*}
by the Minkowski inequality and the assumption $s>2$.
Taking the supremum over all increasing finite sequences we obtain
\begin{equation}
\label{eq:jump-fct-hilbert}
\| f_{l}*\cm(x) \|_{\hV^{s}_{t}(\ell^{2}_{l})}
\leq
\| \| f_{l}*\cm(x) \|_{\hV^{s}_{t}} \|_{\ell^{2}_{l}}.
\end{equation}
Integrating this we obtain
\[
\| \| f_{l}*\cm(x) \|_{\hV^{s}_{t}(\ell^{2}_{l})} \|_{L^{p}_{x}}
\leq
\| \| \| f_{l}*\cm(x) \|_{\hV^{s}_{t}} \|_{\ell^{2}_{l}} \|_{L^{p}_{x}}.
\]
By the Minkowski inequality and the assumption $p\geq 2$ this is bounded by
\[
\| \| \| f_{l}*\cm(x) \|_{\hV^{s}_{t}} \|_{L^{p}_{x}} \|_{\ell^{2}_{l}}.
\]
Using the case $|G|=1$ we obtain the conclusion.
\end{proof}

\subsection{A variation inequality for several frequencies}
A central observation is that Corollaries \ref{cor:NOT:3.2} and \ref{cor:3.30} can be used to show a multi-frequency variation inequality in the same manner as in \cite[Lemma 4.13]{MR1019960}.
\begin{proposition}
\label{prop:4.13}
Let $G\subset\Z^{d}$, $|G|=N$, and $(\xi_{l})_{l\in G}\subset\R^{d}$ be frequencies such that $|\xi_{l,i}-\xi_{l',i}|>|l_{i}-l'_{i}|\tau_{i}$, $\tau_{i}>0$, for all $i=1,\dots,d$ and $l,l'\in G$.
Let $f_{l} \in L^{2}(\R^{d})$, $l\in G$, be functions with $\supp \hat f_{l} \subset [-\tau_{1}/2,\tau_{1}/2]\times\dots\times [-\tau_{d}/2,\tau_{d}/2]$.
Then for any $q>2$ we have
\begin{equation}
\label{eq:prop:4.13:conclusion}
\| \| \sum_{l\in G} e(\xi_{l}\cdot x) (f_{l}*\cm)(x) \|_{\iV^{q}_{t>0}} \|_{L^{2}_{x}(\R^{d})}
\lesssim
\left( \frac{q (\log N+1)}{q-2} \right)^{2} \| \|f_{l}\|_{L^{2}(\R^{d})} \|_{\ell^{2}_{l\in G}}.
\end{equation}
\end{proposition}
\begin{proof}
When the variation norm on the left-hand side of \eqref{eq:prop:4.13:conclusion} is replaced by evaluation at $t=1$, say, the $L^{2}$ bound follows from the Plancherel identity.
Hence it suffices to show \eqref{eq:prop:4.13:conclusion} with the homogeneous variation norm $\hV^{q}$.

It suffices to consider $N>100$ and $q<4$, say.
As in the proof of Corollary~\ref{cor:3.30} we may restrict $t$ in \eqref{eq:prop:4.13:conclusion} to the rationals, and by monotone convergence it suffices to consider a finite subset $T$ of the rationals as long as the bounds are independent of this set.

Let $B_{T}$ be the best constant for which the restricted version of \eqref{eq:prop:4.13:conclusion} on $t\in T$ holds.
It is finite because we can estimate the $\hV^{q}$ norm by the $\ell^{2}$ norm, thereby bounding the left-hand side of \eqref{eq:prop:4.13:conclusion} by
\[
\sum_{l\in G} \| \|f_{l}*\cm\|_{\ell^{2}_{t\in T}} \|_{L^{2}}
\leq \sum_{l\in G} \| \|f_{l}*\cm\|_{L^{2}} \|_{\ell^{2}_{t\in T}}
\leq C \sum_{l\in G} |T|^{1/2} \|f_{l}\|_{L^{2}}
\leq C N^{1/2}|T|^{1/2} \| \|f_{l}\|_{L^{2}} \|_{\ell^{2}_{l\in G}}
\]
using the Minkowski, the Young convolution, and the Hölder inequality.
We now use Bourgain's averaging trick.
Let $R_{u}f(x) = f(x+u)$.
By the frequency support assumption on $f_{l}$ and the Bernstein inequality we have $\|f_{l}-R_{u}f_{l}\|_{L^{2}} \lesssim \sum_{i} \tau_{i} |u_{i}| \|f_{l}\|_{L^{2}}$.
It follows that
\begin{multline*}
\| \| \sum_{l\in G} e(\xi_{l}\cdot x) (f_{l}*\cm)(x) \|_{\hV^{q}_{t\in T}} \|_{L^{2}_{x}}\\
\leq
|I|^{-1} \| \| \int_{I} \sum_{l\in G} e(\xi_{l}\cdot x) (R_{u} f_{l}*\cm)(x) \dif u \|_{\hV^{q}_{t\in T}} \|_{L^{2}_{x}}
+
\frac{B_{T}}{2} \| \|f_{l}\|_{L^{2}} \|_{\ell^{2}_{l\in G}},
\end{multline*}
where $I=I_{1}\times\dots\times I_{d}$, $I_{i}=[-c/\tau_{i},c/\tau_{i}]$, and $c$ is a sufficiently small constant depending only on $d$.
Pulling the integral out of the variation norm and estimating the $L^{1}_{u}$ norm by the $L^{2}_{u}$ norm we obtain the bound
\[
C |I|^{-1/2} \| \| \| \sum_{l\in G} e(\xi_{l}\cdot x) (R_{u} f_{l}*\cm)(x) \|_{\hV^{q}_{t\in T}} \|_{L^{2}_{u}(I)} \|_{L^{2}_{x}}
\]
for the first summand.
The fact that the translation operator $R$ commutes with convolution and a change of variable in the double integral in $x$ and $u$ show that this equals
\[
C |I|^{-1/2} \| \| \| \sum_{l\in G} e(\xi_{l}\cdot (x+u)) (f_{l}*\cm)(x) \|_{\hV^{q}_{t\in T}} \|_{L^{2}_{u}(I)} \|_{L^{2}_{x}}.
\]
By Corollary~\ref{cor:NOT:3.2} for any $2<r<q$ this is bounded by
\begin{multline*}
C \left( \frac{q}{q-r} + \frac{2}{r-2} \right) N^{(\frac12 - \frac1r) \frac{q}{q-2}}
\| \| e(\xi_{l}\cdot x) (f_{l}*\cm)(x) \|_{\hV^{r}_{t\in T}(\ell^{2}_{l\in G})} \|_{L^{2}_{x}}\\
=
C \left( \frac{q}{q-r} + \frac{2}{r-2} \right) N^{(\frac12 - \frac1r) \frac{q}{q-2}}
\| \| (f_{l}*\cm)(x) \|_{\hV^{r}_{t\in T}(\ell^{2}_{l\in G})} \|_{L^{2}_{x}}.
\end{multline*}
By Corollary~\ref{cor:3.30} this is bounded by
\[
C \left( \frac{q}{q-r} + \frac{2}{r-2} \right) N^{(\frac12 - \frac1r) \frac{q}{q-2}} \frac{r}{r-2}
\| \| f_{l} \|_{L^{2}} \|_{\ell^{2}_{l\in G}}.
\]
Choosing $r$ such that $r-2 = (q-2)(\log N)^{-1}$ this gives the bound
\[
C \left( \frac{q \log N}{q-2} \right)^{2} \| \| f_{l} \|_{L^{2}} \|_{\ell^{2}_{l\in G}}.
\]
Hence we have obtained
\[
B_{T} \leq C \left( \frac{q \log N}{q-2} \right)^{2} + \frac{B_{T}}{2},
\]
and the conclusion follows.
\end{proof}

\subsection{$L^{p}$ variation estimates}
Let us now describe the setting in which Proposition~\ref{prop:4.13} will be applied.
Let $\cutoff$ be a Schwartz cut-off function such that $\supp\cutofff\subset [-1/50,1/50]^{d}$ and $\cutofff\equiv 1$ on $[-1/100,1/100]^{d}$.
Let
\[
\rats = \{ (a_{1}/q_{1},\dots,a_{d}/q_{d}) \in \Q^{d}\cap [0,1)^{d} \text{ in reduced form with } 2^{s}\leq \lcm(q_{1},\dots,q_{d})< 2^{s+1}\}
\]
be the set of rational points of height $\approx 2^{s}$ in the unit cube.

Let $\LsR[s,t]$ be Schwartz functions on $\R^{d}$ defined by
\begin{equation}
\label{eq:sL-sN}
\LfR[s,t](\vec\alpha) = \sum_{\vec\theta\in \rats} S(\vec\theta) \cmf (\vec\alpha-\vec\theta) \cutofff[10^{s}](\vec\alpha-\vec\theta),
\end{equation}
where $S(\vec\theta)$ are arbitrary constants and $\cutoff[t](\vec x) = t^{-d} \cutoff(\vec x/t)$ denotes the $L^{1}$-dilation, so that $\cutofff[t](\vec\xi) = \cutofff(t\vec\xi)$.
An important observation is that any two distinct members of $\rats$ are separated at least by $2^{-2s-2}$, so the terms of the sum defining $\LfR[s,t]$ are disjointly supported.
Write
\begin{equation}
\label{eq:Smax}
\Smax = \max_{\vec\theta\in\rats} S(\vec\theta).
\end{equation}
In our applications this quantity will decrease with $s$ sufficiently rapidly to offset the relatively fast growth of the size of the set of frequencies $\rats$, thus making the next result useful at least for $p$ not too far from $2$.
\begin{theorem}
\label{thm:multiple-freq-interpolation}
For any $1<p<\infty$, any $q>2$, and any $\delta>0$ we have
\begin{equation}
\label{eq:multiple-freq-interpolation}
\| \| (\LsR[s,t] * f)(x) \|_{\iV^{q}_{t>0}} \|_{L^{p}_{x}(\R^{d})}
\lesssim_{q,p,\delta}
\Smax |\rats|^{2|\frac1p-\frac12|+\delta} \|f\|_{L^{p}(\R^{d})}.
\end{equation}
\end{theorem}
\begin{proof}
We have
\begin{multline*}
(\LsR[s,t] * f)(x)
=
\sum_{\vec\theta\in\rats} S(\vec\theta) \Finv(\cmf(\cdot-\vec\theta) \cutofff[10^{s}](\cdot-\vec\theta) \hat f)(x)\\
=
\sum_{\vec\theta\in\rats} S(\vec\theta) e(\vec\theta\cdot x) \Finv(\cmf \cutofff[10^{s}] \hat f(\cdot+\vec\theta))(x)\\
=
\sum_{\vec\theta\in\rats} e(\vec\theta\cdot x) (\cm * f_{\vec\theta})(x)
\end{multline*}
with $f_{\vec\theta} = S(\vec\theta) \cutoff[10^{s}] * (e(-\vec\theta\cdot)f)$.
By the Plancherel identity we have
\begin{equation}
\label{eq:f-theta-l2}
\| \|f_{\vec\theta}\|_{L^{2}} \|_{\ell^{2}_{\vec\theta}} \lesssim \Smax \|f\|_{L^{2}},
\end{equation}
whereas by the Young convolution inequality we have
\begin{equation}
\label{eq:f-theta-lp}
\|f_{\vec\theta}\|_{L^{p}} \lesssim |S(\vec\theta)| \|f\|_{L^{p}}
\quad
\text{for any } 1\leq p\leq\infty.
\end{equation}
Using \eqref{eq:f-theta-l2} and Proposition~\ref{prop:4.13} we obtain
\begin{equation}
\label{eq:multiple-freq-good}
\| \| (\LsR[s,t] * f)(x) \|_{\iV^{q}_{t>0}} \|_{L^{2}_{x}}
\lesssim
\Smax \left( \frac{q (\log |\rats|+1)}{q-2} \right)^{2} \|f\|_{L^{2}(\R^{d})}
\lesssim_{q,\delta}
\Smax |\rats|^{\delta} \|f\|_{L^{2}(\R^{d})}
\end{equation}
for any $\delta>0$, where the implied constant does not depend on $s$.
This is the conclusion for $p=2$.

In view of \eqref{eq:multiple-freq-good} and by interpolation it suffices to establish the conclusion with $|\rats|^{2|\frac1p-\frac12|+\delta}$ replaced by $|\rats|$.
This follows from the single-frequency estimate
\[
\| \| \cm * f_{\vec\theta} \|_{\hV^{q}_{t>0}} \|_{L^{p}(\R^{d})}
\lesssim_{q,p}
\| f_{\vec\theta} \|_{L^{p}}
\]
given by Corollary~\ref{cor:3.30a} and the easy estimate at $t=1$.
\end{proof}

\section{Fourier multipliers on $\Z^{d}$}
\label{sec:Z}
\subsection{Transfer from the reals}
Estimates for Fourier multipliers on the real line can be transferred to the integers by a standard averaging argument.
A particularly useful version of that argument, due to Magyar, Stein, and Wainger, shows that the loss in the operator norm is uniformly bounded over all $L^{p}$ spaces and, for operator-valued multipliers, over the Banach spaces in the fibers.
\begin{theorem}[{\cite[Corollary 2.1]{MR1888798}}]
\label{thm:RtoZ}
Let $B_{1},B_{2}$ be finite-dimensional Banach spaces and $m:\R^{d}\to L(B_{1},B_{2})$ a bounded function supported on a cube with side length $1$ containing the origin that acts as a Fourier multiplier from $L^{p}(\R^{d},B_{1})$ to $L^{p}(\R^{d},B_{2})$ for some $1\leq p\leq\infty$.
Let $q$ be a positive integer and
\[
m^{q}_{\mathrm{per}}(\xi) := \sum_{l\in\Z^{d}} m(q\xi-l)
\quad
\text{for } \xi\in (\R/\Z)^{d}.
\]
Then $m^{q}_{\mathrm{per}}$ acts as a Fourier multiplier from $\ell^{p}(\Z^{d},B_{1})$ to $\ell^{p}(\Z^{d},B_{2})$ with norm
\[
\| m^{q}_{\mathrm{per}} \|_{\ell^{p}(\Z^{d},B_{1}) \to \ell^{p}(\Z^{d},B_{2})}
\lesssim_{d}
\| m \|_{L^{p}(\R^{d},B_{1}) \to L^{p}(\R^{d},B_{2})}.
\]
The implied constant does not depend on $p$, $q$, $B_{1}$, and $B_{2}$.
\end{theorem}

This allows us to transfer Theorem~\ref{thm:multiple-freq-interpolation} to the following statement on sequence spaces.
\begin{proposition}
\label{prop:multiple-freq-lp-circle}
Let $1<p<\infty$, $q>2$, and $\delta>0$ be arbitrary.
Then
\begin{equation}
\| \| \Ls[s,t] * f \|_{\iV^{q}_{t>0}} \|_{\ell^{p}(\Z^{d})}
\lesssim_{p,q,\delta}
\Smax |\rats|^{2|\frac1p-\frac12|+\delta} \| f \|_{\ell^{p}(\Z^{d})},
\end{equation}
where $\Lf[s,t]$ is defined by \eqref{eq:sL-sN} as a function on $(\R/\Z)^{d}$.
\end{proposition}
More precisely, note that the convolution on the left-hand side is, pointwise, a continuous function of $t$, so we may restrict attention to rational $t$.
By monotone convergence it suffices to consider a finite set $T$ of $t$'s as long as we obtain an estimate that does not depend on that set.
We apply Theorem~\ref{thm:RtoZ} with $B_{1}=\C$, $B_{2}=(\C^{T},\iV^{q}_{t\in T})$, and $q=1$.

\subsection{Long variation}
In this section we estimate the long variation for convolutions with kernels that admit favorable approximation in terms of
\begin{equation}
\label{eq:LN}
\Ls = \sum_{s\geq 0} \Ls[s,N].
\end{equation}
\begin{theorem}
\label{thm:long-var-lp}
Let $d\geq 1$, $1<p<\infty$, $Z \subset \N$, and suppose
\begin{equation}
\label{eq:K-L}
\sum_{N\in Z} \|\Kf-\Lf\|_{\infty}^{\delta} < \infty
\quad\text{for every }\delta>0
\quad{ and}
\end{equation}
\begin{equation}
\label{eq:S-lp}
\sum_{s\geq 0} |\rats|^{2 |\frac1p-\frac12|+\delta} \Smax < \infty
\quad\text{for some } \delta>0.
\end{equation}
If $p\neq 2$ assume in addition
\begin{equation}
\label{eq:KN-l1}
\sup_{N\in Z} \| \Ks \|_{\ell^{1}} < \infty.
\end{equation}
Then for every $q>2$ we have
\[
\| \|\Ks * f \|_{\iV^{q}_{N\in Z}} \|_{\ell^{p}} \lesssim_{q} \|f\|_{\ell^{p}}.
\]
\end{theorem}
\begin{proof}
We have
\begin{equation}
\label{eq:long-var-lp:1}
\| \|\Ks * f \|_{\iV^{q}_{N\in Z}} \|_{\ell^{p}}
\lesssim
\sum_{s} \| \|\Ls[s,N] * f \|_{\iV^{q}_{N\in Z}} \|_{\ell^{p}}
+
\| \|(\Ks-\Ls) * f \|_{\ell^{1}_{N\in Z}} \|_{\ell^{p}}.
\end{equation}
The first sum in \eqref{eq:long-var-lp:1} is bounded by
\[
\sum_{s} \Smax |\rats|^{2|\frac1p-\frac12|+\delta} \| f \|_{\ell^{p}} \lesssim \|f\|_{\ell^{p}}
\]
by Proposition~\ref{prop:multiple-freq-lp-circle} and \eqref{eq:S-lp}.
Note that this estimate also holds for some $1<p_{0}<\infty$ which is farther away from $2$ than $p$ and implies in particular
\[
\| \Ls \|_{\ell^{p_{0}}\to \ell^{p_{0}}} \leq C,
\quad N\in Z.
\]
Since for our kernels also
\[
\| \Ks \|_{\ell^{p_{0}}\to \ell^{p_{0}}} \leq \| \Ks \|_{\ell^{1}} \leq C
\]
by \eqref{eq:KN-l1}, we obtain
\[
\| \Ks - \Ls \|_{\ell^{p_{0}}\to\ell^{p_{0}}} \leq C.
\]
Interpolating between $\ell^{p_{0}}$ and $\ell^{2}$ we obtain
\[
\| \Ks - \Ls \|_{\ell^{p}\to\ell^{p}}
\leq
\| \Ks - \Ls \|_{\ell^{p_{0}}\to\ell^{p_{0}}}^{1-\theta} \| \Ks - \Ls \|_{\ell^{2}\to\ell^{2}}^{\theta}
\lesssim
\| \Kf - \Lf \|_{\infty}^{\theta},
\]
where $\theta>0$ is obtained from the condition $1/p=(1-\theta)/p_{0}+\theta/2$.
Interpolation is not needed if $p=2$, and the condition \eqref{eq:KN-l1} is consequently not used in that case.
The last term is summable in $N$ by \eqref{eq:K-L}.
This, and Minkowski's inequality, allows us to estimate the second term in \eqref{eq:long-var-lp:1}.
\end{proof}

\subsection{Short variation}
Since we will be able to handle the long variation on fairly dense subsets of $\N$, namely
\begin{equation}
\label{eq:Z_epsilon}
Z=Z_{\epsilon}=\{\lfloor 2^{k^{\epsilon}} \rfloor\}_{k=1}^{\infty},
\end{equation}
we can afford estimating the short variation in a very simplistic manner.
\begin{lemma}
\label{lem:short-var-lp}
Let $1<p<\infty$, $1<q\leq \infty$, let $Z = \{N_{1},N_{2},\dots\}\subset \N$ be an increasing sequence, and suppose
\begin{equation}
\label{eq:K-l1}
\| \sum_{N\in\Nker_{k}}\|K_{N+1}-\Ks\|_{\ell^{1}} \|_{\ell^{\min(p,q)}_{k}} < \infty,
\end{equation}
where $\Nker_{k} = [N_{k},N_{k+1}]$.
Then
\[
\| \| \Ks*f \|_{\sV^{q}_{N}} \|_{\ell^{p}} \lesssim_{p,q} \|f\|_{\ell^{p}}.
\]
\end{lemma}
\begin{proof}
In view of the monotonicity of the variation norms it suffices to consider $q\leq p$.
By the definition of the short variation norm, the monotonicity of variation norms, and two applications of the Minkowski inequality we have
\begin{multline}
\label{eq:short-var:1}
\| \| \Ks*f \|_{\sV^{q}_{N}} \|_{\ell^{p}}
=
\| \| \|\Ks*f\|_{\hV^{q}_{N\in\Nker_{j}}} \|_{\ell^{q}_{j}} \|_{\ell^{p}}
\leq
\| \| \|\Ks*f\|_{\hV^{1}_{N\in\Nker_{j}}} \|_{\ell^{q}_{j}} \|_{\ell^{p}}\\
\leq
\| \| \|(K_{N+1}-\Ks)*f\|_{\ell^{1}_{N\in\Nker_{j}}} \|_{\ell^{q}_{j}} \|_{\ell^{p}}
\leq
\| \| \|(K_{N+1}-\Ks)*f\|_{\ell^{p}} \|_{\ell^{1}_{N\in\Nker_{j}}} \|_{\ell^{q}_{j}}.
\end{multline}
By the Young convolution inequality this is bounded by
\[
\| \| \|K_{N+1}-\Ks\|_{\ell^{1}} \|_{\ell^{1}_{N\in\Nker_{j}}} \|_{\ell^{q}_{j}} \|f\|_{\ell^{p}}
\]
By the hypothesis \eqref{eq:K-l1} this is $\lesssim \|f\|_{\ell^{p}}$.
\end{proof}

\begin{corollary}
\label{cor:full-var}
Let $1<p<\infty$ and suppose that for some $Z\subset\N$ the conditions \eqref{eq:K-L}, \eqref{eq:S-lp}, and \eqref{eq:K-l1} hold.
If $p\neq 2$ assume in addition \eqref{eq:KN-l1}.
Then for every $q>2$ we have
\[
\| \|\Ks * f \|_{\hV^{q}_{N}} \|_{\ell^{p}} \lesssim_{p,q} \|f\|_{\ell^{p}}.
\]
\end{corollary}
\begin{proof}
By Theorem~\ref{thm:long-var-lp} and Lemma~\ref{lem:short-var-lp} we have $\ell^{p}$ bounds for $\|\Ks * f \|_{\hV^{q}_{N\in Z}}$ and $\| \Ks * f\|_{\sV^{q}_{N}}$ with respect to $Z$.
The conclusion follows from the pointwise bound \eqref{eq:var-long-short}.
\end{proof}

\section{Primes}
\label{sec:primes}
In this section we recall several estimates from \cite[\textsection 4]{MR950982} and \cite{MR995574}, partially following the exposition in \cite{2013arXiv1311.7572M}, and prove Theorem~\ref{thm:var-primes}.

\subsection{Tools}
We begin with the necessary tools from number theory.
In this section $\Lambda$ denotes the von Mangoldt function, $\mu$ the Möbius function, and $\tot$ the Euler totient function.
For $q\in\N$ let $A_{q}=\{r\in\{1,\dots,q\} : (r,q)=1\}$.
Recall the Ramanujan sum identity \cite[Theorem A.24]{MR1395371}
\begin{equation}
\label{eq:ramanujan-sum}
\sum_{r\in A_{q}} e(ra/q) = \mu(q),
\quad (a,q)=1.
\end{equation}
and the elementary estimate
\begin{equation}
\label{eq:totient-lower-bound}
\tot(n) \gtrsim_{\delta} n^{1-\delta}
\quad \text{ for any }\delta>0
\end{equation}
for the Euler totient function, see \cite[Theorem A.16]{MR1395371}.
\begin{theorem}[{Vinogradov, see \cite[\textsection 25]{MR606931}}]
\label{thm:vinogradov}
Suppose $|\alpha-a/q| \leq 1/q^{2}$, $a\in A_{q}$.
Then
\[
\Big| \sum_{n\leq N}\Lambda(n)e(n\alpha) \Big|
\lesssim
(Nq^{-1/2}+N^{4/5}+N^{1/2}q^{1/2})(\log N)^{4}.
\]
\end{theorem}

\begin{theorem}[{Siegel--Walfisz, see \cite[\textsection 22]{MR606931}}]
\label{lem:siegel-walfisz}
Let
\[
\psi(N;q,r) = \sum_{n\leq N, n\equiv r \mod q} \Lambda(n)
\]
be the (von Mangoldt weighted) counting function for the primes $\equiv r\mod q$.
For every $A>0$ there exists $C(A)>0$ such that
\[
\psi(N;q,r) = \frac{N}{\tot(q)} + O(N \exp(-C(A) \sqrt{\log N}))
\]
for every $q\leq (\log N)^{A}$ and $r\in A_{q}$.
\end{theorem}

\subsection{Approximation of the kernel}
\begin{lemma}
\label{lem:approx-exp-sum-lambda}
Let $A>0$, $0<\epsilon<1$, and $\alpha = a/q + \beta$, where $q\leq ((1-\epsilon)\log N)^{A}$ and $|\beta|<(\log N)^{A}/N$.
Then
\[
\big| \Kf(\alpha) - \frac{\mu(q)}{\tot(q)} \frac1N \sum_{n\leq N} e(n\beta) \big|
\lesssim_{A}
e^{-C(A)(1-\epsilon) \sqrt{\log N}},
\]
where $C(A)$ is the constant from the Siegel--Walfisz theorem.
\end{lemma}
The partial summation argument below is adapted from \cite{2013arXiv1311.7572M}.
\begin{proof}
We write the von Mangoldt function as the increment of the weighted prime counting function
\[
\Lambda(n)
=
\sum_{r=1}^{q} (\psi(n;q,r) - \psi(n-1;q,r)).
\]
The terms with $r\not\in A_{q}$ are non-zero only for those $n$ that are powers of the primes that divide $q$, and there are at most $q \log N$ such $n$'s.
Therefore
\begin{multline}
\label{eq:Kf-psi-inc}
\Kf(\alpha)
=
\frac1N \sum_{n\leq N} \Lambda(n) e(n\alpha)\\
=
\frac1N \sum_{r\in A_{q}} e(ra/q) \sum_{n=M+1}^{N} (\psi(n;q,r)-\psi(n-1;q,r)) e(n\beta)
+ O(\frac{(\log N)^{A+2}}{N}) + O(M/N).
\end{multline}
With $M=\lceil N^{1-\epsilon} \rceil$ both error terms can be absorbed into the error term of the conclusion. 
Now consider the sum over $n$ in the main term.
By partial summation it equals
\begin{equation}
\label{eq:Kf-psi-inc:main}
\psi(N;q,r)e(N\beta)-\psi(M;q,r)e(M\beta) + \sum_{n=M+1}^{N-1} \psi(n;q,r)(e(n\beta)-e((n+1)\beta)).
\end{equation}
We use Theorem~\ref{lem:siegel-walfisz} to split this into a main term and the error term.
The main term equals
\[
\frac{N}{\tot(q)} e(N\beta) - \frac{M}{\tot(q)}e(M\beta)
+ \sum_{n=M+1}^{N-1} \frac{n}{\tot(q)}(e(n\beta)-e((n+1)\beta))
=
\frac{1}{\tot(q)} \sum_{n=M}^{N} e(n\beta)
\]
by partial summation.
Summing up the contributions of these terms to \eqref{eq:Kf-psi-inc} we obtain
\[
\frac1N \sum_{r\in A_{q}} e(ra/q) \frac1{\tot(q)} \sum_{n=M}^{N} e(n\beta)
=
\frac{\mu(q)}{\tot(q)} \frac1N \sum_{n\leq N} e(n\beta)
+ O(M/N)
\]
by the Ramanujan sum identity \eqref{eq:ramanujan-sum}.
It remains to estimate the error term produced by application of Theorem~\ref{lem:siegel-walfisz} to \eqref{eq:Kf-psi-inc:main}.
It equals
\begin{multline*}
O(Ne^{-C(A)\sqrt{\log N}})+O(Me^{-C(A)\sqrt{\log M}})
+ \sum_{n=M+1}^{N-1} O(ne^{-C(A)\sqrt{\log n}})|e(n\beta)-e((n+1)\beta)|\\
=
O(Ne^{-C(A)\sqrt{\log M}})
+ |\beta| \sum_{n=M+1}^{N-1} O(N e^{-C(A)\sqrt{\log M}})\\
=
O(Ne^{-C(A)\sqrt{\log M}})
+ \frac{(\log N)^{A}}{N} O(N^{2} e^{-C(A)\sqrt{\log M}})\\
=
O((\log N)^{A}N e^{-C(A)\sqrt{1-\epsilon}\sqrt{\log N}}).
\end{multline*}
The contribution of this term to \eqref{eq:Kf-psi-inc} can therefore be estimated by
\[
q (\log N)^{A} e^{-C(A)\sqrt{1-\epsilon}\sqrt{\log N}}
\lesssim
e^{-C(A)(1-\epsilon)\sqrt{\log N}}.
\qedhere
\]
\end{proof}

\begin{lemma}[{\cite[(6)]{MR995574}}]
\label{lem:K-L:primes}
Let $\Ks$ be given by \eqref{eq:KN:primes} and let $\Ls$ be given by \eqref{eq:LN} with $d=1$ and
\begin{equation}
\label{eq:S-prime}
S(\theta)=\frac{\mu(q)}{\tot(q)}
\quad\text{for } \theta=a/q \text{ in reduced form}.
\end{equation}
Then for every $B>0$ we have
\[
\|\Kf - \Lf\|_{\infty} \lesssim_{B} (\log N)^{-B}.
\]
\end{lemma}
\begin{proof}
Let $A$ be sufficiently large ($A=3B+8$ will do) and assume without loss of generality that $N$ is sufficiently large depending on $A$.
Let $\alpha\in [0,1]$.
By Dirichlet's approximation theorem there exists a reduced fraction $a/q$ such that
\begin{equation}
\label{eq:K-L:primes:Dirichlet}
\left| \alpha - \frac{a}{q} \right| \leq \frac1{qQ}
\quad\text{with}\quad q\leq Q = N(\log N)^{-A}.
\end{equation}
Let $s_{0}$ be such that $a/q \in \rats[s_{0}]$ and let $\beta = \alpha - a/q$.
For each $s\neq s_{0}$ there is at most one $a_{s}/q_{s}\in\rats$ that contributes to the sum defining $\Lf[s,N](\alpha)$, and
\begin{equation}
\label{eq:Ls-neq-s0}
\sum_{s\neq s_{0}} |\Lf[s,N](\alpha)|
\leq
\sum_{s\neq s_{0}} \Smax \cmf[N](\alpha-a_{s}/q_{s})
\lesssim
\sum_{s\neq s_{0}} \Smax (1+N|\alpha-a_{s}/q_{s}|)^{-1}.
\end{equation}
We have
\begin{equation}
\label{eq:alpha-as-qs}
\left| \alpha - \frac{a_{s}}{q_{s}} \right|
\geq
\left|\frac{a_{s}}{q_{s}}-\frac{a}{q}\right| - |\beta|
\geq
\frac{1}{qq_{s}} - \frac{1}{qQ}.
\end{equation}
The first term dominates if $2^{s} < Q/4$ and in particular if $2^{s} \lesssim_{A} (\log N)^{A/2}$.
Thus we may estimate \eqref{eq:Ls-neq-s0} by
\[
\sum_{s : 2^{s} \lesssim_{A} (\log N)^{A/2}} N^{-1} q \cdot 2^{s} + \sum_{s : 2^{s} \gtrsim_{A} (\log N)^{A/2}} \Smax.
\]
Using \eqref{eq:K-L:primes:Dirichlet} in the first term and \eqref{eq:totient-lower-bound} in the second term we see that this is $\lesssim (\log N)^{-A/2+\delta}$.
Thus it remains to show
\[
| \Lf[s_{0},N](\alpha) - \Kf(\alpha) | \lesssim_{B} (\log N)^{-B}.
\]

\paragraph{Major arcs.}
Suppose that $q\leq (\frac12 \log N)^{A}$.
The estimate \eqref{eq:alpha-as-qs} shows that $|\alpha-\theta|>4^{-s_{0}}/2$ for all $\theta\in\rats[s_{0}]\setminus\{a/q\}$ (provided that $N$ is large enough), so the corresponding terms in $\Lf[s_{0},N]$ vanish at $\alpha$, and we obtain
\[
\Lf[s_{0},N](\alpha) = S(a/q) \cmf[N](\beta) \cutofff[10^{s_{0}}](\beta).
\]
Further,
\[
|\cutofff[10^{s_{0}}](\beta)-1|
\lesssim 10^{s_{0}}|\beta|
\lesssim q^{4} \frac{1}{qQ}
\lesssim \frac{(\log N)^{4A}}{N}. 
\]
Finally,
\[
\Big| \cmf[N](\beta) - \frac1N \sum_{n\leq N}e(n\beta) \Big|
\lesssim
|\beta|
\leq
\frac{(\log N)^{A}}{N}.
\]
Combining these estimates with Lemma~\ref{lem:approx-exp-sum-lambda} we obtain the desired bound at $\alpha$.

\paragraph{Minor arcs.}
Suppose now $q>(\frac12 \log N)^{A}$.
Then
\[
|\Lf[s_{0},N](\alpha)| \leq \Smax[s_{0}] \lesssim q^{-1+\delta} \lesssim (\log N)^{(-1+\delta)A}.
\]
On the other hand, by Theorem~\ref{thm:vinogradov} we have
\begin{multline*}
|\Kf(\alpha)|
\lesssim
(q^{-1/2}+N^{-1/5}+N^{-1/2}q^{1/2})(\log N)^{4}\\
\lesssim
((\log N)^{-A/2}+N^{-1/5})(\log N)^{4}
\lesssim
(\log N)^{-A/2+4}.
\qedhere
\end{multline*}
\end{proof}

\begin{proof}[Proof of Theorem~\ref{thm:var-primes} for $\big| \frac1p - \frac12 \big| < \frac{1}{4}$]
It suffices to verify the conditions of Corollary~\ref{cor:full-var} with $d=1$.
We consider $Z=Z_{\epsilon}$ as in \eqref{eq:Z_epsilon} with a sufficiently small $\epsilon$ to be chosen shortly.
We have
\[
\|K_{N+1}-\Ks\|_{\ell^{1}} \leq \frac{1+\log(N+1)}{N+1}.
\]
For $\Nker_{k} = [N_{k},N_{k+1}]$ we also have
\[
|\Nker_{k}| \lesssim 2^{k^{\epsilon}}(2^{(k+1)^{\epsilon}-k^{\epsilon}}-1) \lesssim 2^{k^{\epsilon}} k^{-1+\epsilon}
\]
This gives the bound
\[
\| 2^{k^{\epsilon}} k^{-1+\epsilon} (1+\log((2^{(k+1)^{\epsilon}})+1))/2^{k^{\epsilon}} \|_{\ell^{\min(p,q)}_{k}}
\lesssim
\| k^{-1+2\epsilon} \|_{\ell^{\min(p,q)}_{k}}
\]
for \eqref{eq:K-l1}.
This is finite provided $(1-2\epsilon) \min(p,q) > 1$.

The condition \eqref{eq:KN-l1} is immediate.

With the choice \eqref{eq:S-prime} we have
\[
\Smax \lesssim 2^{-s(1-\delta)}
\]
for any $\delta>0$ by \eqref{eq:totient-lower-bound}.
Since $|\rats| \leq 4^{s}$, this gives \eqref{eq:S-lp} whenever $\big| \frac1p - \frac12 \big| < \frac{1}{4}$.
Finally, \eqref{eq:K-L} follows from Lemma~\ref{lem:K-L:primes}.
\end{proof}

\subsection{A multi-frequency estimate for $1<p<\infty$}
In view of the results in \cite{MR2653686} it appears plausible that the exponent $2\big| \frac1p - \frac12 \big| + \delta$ in Proposition~\ref{prop:4.13} can be improved to $\big| \frac1p - \frac12 \big|+\delta$.
If this is indeed the case, then the above proof immediately gives Theorem~\ref{thm:var-primes} for $1<p<\infty$.

As an aside we note that the required improvement of Proposition~\ref{prop:4.13} can be obtained for $p>2$ by proving a version of Corollary~\ref{cor:3.30a} for Hilbert space-valued functions, which seems to be possible using the methods in \cite{MR2434308}, and applying it together with Rubio de Francia's Littlewood--Paley inequality for arbitrary intervals \cite{MR850681} to obtain the necessary endpoint estimates for $p$ near $\infty$.

However, this approach does not extend to $p<2$, and, in any case, a much simpler argument due to Wierdl \cite{MR995574} works for our purposes.
The main additional ingredient is the Magyar--Stein--Wainger result on periodic multipliers (Theorem~\ref{thm:RtoZ}), which can be used to close a gap on p.\ 331 in \cite{MR995574}: there, the proof of the estimate for (**) gives $q^{p}$ instead of $q$.
(That gap has been already closed in \cite{2013arXiv1311.7572M} using what amounts to a special case of Theorem~\ref{thm:RtoZ}.
We note that this gap does not affect the validity of the results in \cite{MR995574} since the discrepancy between $q$ and $q^{p}$ can be absorbed into the estimates that follow for $p$ near $1$ and interpolation with the easy endpoint at $p=\infty$ allows one to handle large values of $p$.
However, the situation for variational estimates is different due to lack of such an easy endpoint, and the full power of Theorem~\ref{thm:RtoZ} is useful here.)
\begin{lemma}
\label{lem:var-all-q}
For $q\in\N$ and $Q\in\R_{>0}$ let
\[
\Lf[q,Q,t](\vec\alpha)
:=
\sum_{a_{1},\dots,a_{d}=1}^{q} \cmf(\vec\alpha-\frac aq) \cutofff[Q](\vec\alpha-\frac aq).
\]
Suppose $q\leq 25 Q$, $r>2$, $1<p<\infty$.
Then
\[
\| \| \Ls[q,Q,t] * f \|_{\iV^{r}_{t>0}} \|_{\ell^{p}(\Z^{d})}
\lesssim_{p,r}
\| f \|_{\ell^{p}(\Z^{d})}.
\]
\end{lemma}
\begin{proof}
As usually, it suffices to consider a finite set $T$ of $t$'s as long as the bound does not depend on $T$.
The estimate follows from the single-frequency estimate in Corollary~\ref{cor:3.30a} using Theorem~\ref{thm:RtoZ} with $B_{1}=\C$, $B_{2}=(\C^{T},\iV^{r}_{t\in T})$, $q=q$, and $p=p$.
\end{proof}

\begin{corollary}
\label{cor:var-A-q}
Let
\[
\Lf[A_{q}^{d},Q,t](\vec\alpha)
:=
\sum_{a\in A_{q}^{d}} \cmf(\vec\alpha-\frac aq) \cutofff[Q](\vec\alpha-\frac aq),
\]
where $A_{q}^{d}=\{(a_{1},\dots,a_{d}) : (a_{1},\dots,a_{d},q)=1\}$.
Suppose $q\leq 25 Q$, $q\in\N$, $Q\in\R_{>0}$, $r>2$, $1<p<\infty$.
Then
\[
\| \| \Ls[q,Q,t] * f \|_{\iV^{r}_{t>0}} \|_{\ell^{p}(\Z^{d})}
\lesssim_{p,r,\epsilon}
q^{\epsilon} \| f \|_{\ell^{p}(\Z^{d})}
\]
for any $\epsilon>0$.
\end{corollary}
\begin{proof}
Recall \cite[Theorem 266]{MR2445243} the Möbius inversion formula
\[
g(n) = \sum_{d|n} f(d)
\implies
f(n) = \sum_{d|n} \mu(\frac nd) g(d).
\]
It is applied with $f(n) = \Ls[A_{n}^{d},Q,t]$, then $g(n) = \Ls[n,Q,t]$, so
\[
\Ls[A_{q}^{d},Q,t] = \sum_{d|q} \mu(\frac qd) \Ls[d,Q,t].
\]
For each summand we have the estimate $\lesssim \|f\|_{\ell^{p}} \leq \|f\|_{\ell^{p}}$ by Lemma~\ref{lem:var-all-q} and there are $O(q^{\epsilon})$ summands since $q$ has $O(q^{\epsilon})$ divisors for any $\epsilon>0$ \cite[Theorem 315]{MR2445243}.
\end{proof}

\begin{corollary}
\label{cor:var-rats}
Let $d=1$ and $S(a/q)=\mu(q)/\tot(q)$.
Suppose $r>2$, $1<p<\infty$.
Then
\[
\| \| \Ls[s,t] * f \|_{\iV^{r}_{t\in\R}} \|_{\ell^{p}} \lesssim_{p,r,\epsilon} 2^{s\epsilon} \| f \|_{\ell^{p}}
\]
for any $\epsilon>0$.
\end{corollary}
\begin{proof}
We have
\[
\Ls[s,t] = \sum_{q=2^{s}}^{2^{s+1}-1} \frac{\mu(q)}{\tot(q)} \Ls[A_{q},10^{s},t],
\]
so
\[
\| \| \Ls[s,t] * f \|_{\iV^{r}_{t\in\R}} \|_{\ell^{p}}
\leq
\sum_{q=2^{s}}^{2^{s+1}-1} \frac{1}{\tot(q)} \| \| \Ls[A_{q},10^{s},t] * f \|_{\iV^{r}_{t\in\R}} \|_{\ell^{p}}.
\]
By Corollary~\ref{cor:var-A-q} with $d=1$ and \eqref{eq:totient-lower-bound} this is bounded by
\[
\sum_{q=2^{s}}^{2^{s+1}-1} \frac{1}{q^{1-\epsilon}} q^{\epsilon} \| f \|_{\ell^{p}}
\lesssim
(2^{s})^{2\epsilon} \| f \|_{\ell^{p}}
\]
for any $\epsilon>0$.
\end{proof}

\begin{proof}[Proof of Theorem~\ref{thm:var-primes} for $1<p<\infty$]
An inspection of \textsection\ref{sec:Z} reveals that it suffices to prove a bound for
\[
\| \| \Ls[s,t] * f \|_{\iV^{r}_{t\in\R}} \|_{\ell^{p}}
\]
that is summable is $s$.
This is given by interpolation between Corollary~\ref{cor:var-rats} and Proposition~\ref{prop:multiple-freq-lp-circle}.
\end{proof}

\section{Polynomials}
\label{sec:poly}
In order to make the presentation self-contained we summarize here the approximation of the kernel \eqref{eq:KN:poly} in terms of objects introduced in \textsection\ref{sec:Z} following the argument in \cite{MR1019960}.
We then prove Theorem~\ref{thm:var-poly}.

We denote by $\delta>0$ a small quantity that depends only on $d$.
This symbol's value may change between its uses.
Also, the implied constant in the $\lesssim$ notation may depend on $d$; all other dependencies will be noted explicitly.

\subsection{Estimates for trigonometric sums}
Both the minor and the major arc estimates in this section rely on the following estimate for complete exponential sums, which is due to Hua.
\begin{lemma}[{\cite{MR0004259}}]
\label{lem:hua}
Let $\theta_{j}=a_{j}/q_{j}$, $\lcm(q_{1},\dots,q_{d})=N$.
Then for any $\delta>0$ we have
\[
|\Kf(\vec\theta)| \lesssim_{d,\delta} N^{-1/d+\delta}.
\]
\end{lemma}
The major arcs used in the estimate for $\Kf$ are
\begin{equation}
\label{eq:major:poly}
\mathfrak{M}_{N} = \{ \vec\alpha\in\R^{d} : \alpha_{j} = a_{j}/q_{j}+\beta_{j}, (a_{j},q_{j})=1, |\beta_{j}|\leq N^{-j+\nu}, 1\leq j\leq d, \lcm(q_{1},\dots,q_{d})\leq N^{\nu}\},
\end{equation}
where $\nu=1/\max(d,12)$.
On the minor arcs we have the following estimate.
\begin{lemma}[{\cite[Chapter IV, Theorem 3 (p.\ 72)]{MR0409380}}]
\label{lem:5.6}
There exists $\delta=\delta(d)>0$ such that for any $N$ and any $\vec\alpha\not\in \mathfrak{M}_{N}$ we have
\[
|\Kf(\vec\alpha)| \lesssim N^{-\delta}.
\]
\end{lemma}
In Vinogradov's 1971 book (which is almost, but not quite, entirely unlike the homonymous 1947 publication and its 1954 English translation) this result is stated for $d\geq 12$, which clearly implies the cases $d<12$.

\subsection{Approximation of trigonometric sums}
\begin{lemma}[{\cite[Lemma 5.12]{MR1019960}}]
\label{lem:5.12}
Let $\vec\alpha \in \mathfrak{M}_{N}$.
Then
\[
\Kf(\vec\alpha) = \Kf[q](\vec\theta) \cmf[N](\vec\beta) + O(N^{-\delta})
\quad\text{ for some } \delta>0,
\]
where $q = \lcm(q_{1},\dots,q_{d})$ and $\vec\beta$ is as in \eqref{eq:major:poly}.
\end{lemma}
\begin{proof}
For $n=qs+r$, $1\leq s\leq \lfloor N/q \rfloor$, $1\leq r\leq q$, it follows from the assumptions that
\[
\alpha_{j}n^{j} = \theta_{j}r^{j} + \beta_{j}s^{j}q^{j} + O(N^{-\delta}).
\]
Hence the average in the definition of $\Kf(\vec\alpha)$ can be approximated by a product of averages over $s$ and $r$, namely
\[
\Kf(\vec\alpha) = \Kf[q](\vec\theta) \Kf[\lfloor N/q \rfloor](\beta_{1}q^{1},\dots,\beta_{d}q^{d}) + O(N^{-\delta}).
\]
Now, $\Kf[\lfloor N/q \rfloor](\beta_{1}q^{1},\dots,\beta_{d}q^{d})$ is a Riemann sum for $\cmf[N](\vec\beta)$ at scale $q$.
The assumptions on $|\beta_{j}|$ imply that the derivative of the integrand in $\cmf[N](\vec\beta)$ is $O(N^{-\delta})$, and the boundary effects are of the same order, so the error incurred by passing from the sum to the integral is $O(N^{-\delta})$.
\end{proof}

\begin{lemma}[{\cite[Lemma 5.17]{MR1019960}}]
For every $1\leq j\leq d$ we have
\begin{equation}
\label{eq:V-oscillatory}
|\cmf[N](\vec\beta)| \lesssim_{j,d} |\beta_{j}|^{-1/d}N^{-j/d}.
\end{equation}
\end{lemma}
\begin{proof}
Recall
\[
\cmf[N](\vec\beta) = \frac1N \int_{s=0}^{N} e(\phi(x)) \dif x,
\]
where $\phi(x) = \beta_{1}x^{1} + \dots + \beta_{d}x^{d}$.
We will show \eqref{eq:V-oscillatory} by descending induction on $j$.
Suppose that \eqref{eq:V-oscillatory} is known for all $j>k$, we have to show \eqref{eq:V-oscillatory} with $j=k$.
If $|\beta_{j}|^{-1/d}N^{-j/d} \lesssim |\beta_{k}|^{-1/d}N^{-k/d}$ for some $j>k$, then this follows from the induction hypothesis.
Otherwise we have $|\beta_{j}|N^{j} \lesssim |\beta_{k}|N^{k}$, so that
\[
|\phi^{(k)}(x)|
=
\big| k! \beta_{k} + \sum_{j=k+1}^{d} \frac{j!}{(j-k)!} \beta_{j} x^{j-k} \big|
>
|k! \beta_{k}| - \sum_{j=k+1}^{d} \frac{j!}{(j-k)!} |\beta_{j}| N^{j-k}
\gtrsim
|\beta_{k}|
\]
for all $x\in [0,N]$ provided that the implied constant in the $\lesssim$ notation was chosen sufficiently small.
The van der Corput estimate \cite[\textsection VIII.1.2, Proposition 2]{MR1232192} then implies the desired conclusion (note that the van der Corput estimate is applicable also for $k=1$ since $\phi''$ changes sign at most $d$ times).
\end{proof}

\begin{proposition}[{cf.\ \cite[Lemma 6.14]{MR1019960}}]
\label{prop:6.14}
Let $\Ls$ be given by \eqref{eq:LN} with
\begin{equation}
\label{eq:S-poly}
S(\vec\theta) = \Kf[q](\vec\theta),
\quad
\text{ where } q \text{ is the least common denominator of } \vec\theta.
\end{equation}
and let $\Ks$ be given by \eqref{eq:KN:poly}.
Then
\[
\|\Kf - \Lf\|_{\infty} \lesssim N^{-\delta}
\quad\text{ for some } \delta>0.
\]
\end{proposition}
\begin{proof}
We may assume $N>10000$, say.
Suppose first $\vec\alpha \in \mathfrak{M}_{N}$.
Let $\theta_{j}=a_{j}/q_{j}$ and $\beta_{j}=\alpha_{j}-\theta_{j}$ be as in \eqref{eq:major:poly}.
By Lemma \ref{lem:5.12} we have
\[
\Kf(\vec\alpha) = S(\vec\theta)\cmf[N](\vec\beta) + O(N^{-\delta}).
\]
Let $s_{0}$ be such that $\vec\theta \in \rats[s_{0}]$ and let $s_{1}$ be the integer such that $N^{\nu} \leq 2^{s_{1}} < 2N^{\nu}$.
Then
\[
|S(\vec\theta)\cmf[N](\vec\beta) - \sum_{s\geq 0} \Lf[s,N](\vec\alpha)|
\lesssim
|1-\cutofff[10^{s_{0}}](\vec\beta)| + \sum_{0 \leq s \leq s_{1}} \sup_{\vec\theta'\in \rats, \vec\theta'\neq\vec\theta} |\cmf[N](\vec\alpha-\vec\theta')| + \sum_{s>s_{1}} \sup_{\vec\theta'\in \rats} |S(\vec\theta')|.
\]
Since $|\beta_{j}| \leq N^{-j+\nu} \leq N^{-4\nu}/100 \leq 10^{-s_{0}}/100$, the first term vanishes.
We estimate the third term using Lemma~\ref{lem:hua} and the second term using \eqref{eq:V-oscillatory} and the observation that for $\vec\theta'\in \rats$, $s\leq s_{1}$, we have
\[
|\alpha_{j}-\theta'_{j}|
\geq |\theta'_{j}-\theta_{j}| - |\alpha_{j}-\theta_{j}|
\gtrsim 2^{-s_{1}-s_{0}} - N^{-j+\nu}
\gtrsim N^{-2\nu}
\]
for any $j$ with $\theta'_{j} \neq \theta_{j}$.
This yields the estimate
\[
|\Kf(\vec\alpha) - \sum_{s\geq 0} \Lf[s,N](\vec\alpha)|
\lesssim
N^{-\delta} + \sum_{0 \leq s \leq s_{1}} (N^{-2\nu} \cdot N)^{-1/d} + \sum_{s>s_{1}} (2^{s})^{-\delta}
\lesssim
N^{-\delta}
\]
as required.

If $\vec\alpha \not\in \mathfrak{M}_{N}$, then we have $|\Kf(\vec\alpha)| \lesssim N^{-\delta}$ by Lemma~\ref{lem:5.6}.
Let $s_{1}$ be the integer such that $N^{\nu}/2 \leq 2^{s_{1}} < N^{\nu}$.
We have
\[
|\sum_{s\geq 0} \Lf[s,N](\vec\alpha)|
\lesssim
\sum_{0 \leq s < s_{1}} \sup_{\vec\theta\in \rats} |\cmf[N](\vec\alpha-\vec\theta)| + \sum_{s\geq s_{1}} \sup_{\vec\theta\in \rats} |S(\vec\theta)|.
\]
The second summand is $\lesssim N^{-\delta}$ as before.
In the first summand note that $|\alpha_{j}-\theta_{j}| > N^{-j+\nu}$ for some $j$ since otherwise we would have $\vec\alpha \in \mathfrak{M}_{N}$.
In view of \eqref{eq:V-oscillatory} this implies
\[
|\cmf[N](\vec\alpha-\vec\theta)| \lesssim N^{-\nu/d},
\]
and, since the first summand above consists of approximately $\log N$ terms, we are done.
\end{proof}

\begin{proof}[Proof of Theorem~\ref{thm:var-poly} for $\big| \frac1p - \frac12 \big| < \frac{1}{4d^{2}}$]
It suffices to verify the conditions of Corollary~\ref{cor:full-var} with $Z=Z_{\epsilon}$ as in \eqref{eq:Z_epsilon} for a sufficiently small $\epsilon$.
Verification of \eqref{eq:K-l1} is similar to, but easier than, the proof of Theorem~\ref{thm:var-primes}, and the condition \eqref{eq:KN-l1} is trivially satisfied.

By Lemma~\ref{lem:hua} we have
\[
\Smax \lesssim_{\delta} 2^{s(-1/d+\delta)}
\quad\text{for any } \delta>0,
\]
and since $|\rats|\leq 4^{(s+1)d}$ this implies \eqref{eq:S-lp} provided $\big| \frac1p - \frac12 \big| < \frac{1}{4d^{2}}$.
Finally, condition \eqref{eq:K-L} is given by Proposition~\ref{prop:6.14}.
\end{proof}

\subsection{Another multi-frequency estimate for $1<p<\infty$}
The difficulty in the case of polynomials stems from the fact that the estimates for the exponential sums $S(\vec a/q)$ are much worse than in the case of the primes.
We have the following version of Lemma~\ref{lem:var-all-q}
\begin{lemma}
\label{lem:var-all-q:poly}
For $q\in\N$ and $Q\in\R_{>0}$ let
\[
\Lf[q,Q,t,\mathrm{poly}] (\vec\alpha)
:=
\sum_{a_{1},\dots,a_{d}=1}^{q} S(a/q) \cmf(\vec\alpha-\frac aq) \cutofff[Q](\vec\alpha-\frac aq),
\]
where the coefficients $S$ are given by \eqref{eq:S-poly}.
Suppose $q\leq 5Q$, $r>2$, $1<p<\infty$.
Then
\[
\| \| \Ls[q,Q,t,\mathrm{poly}] * f \|_{\iV^{r}_{t\in\R}} \|_{\ell^{p}(\Z^{d})} \lesssim_{p,r} \| f \|_{\ell^{p}(\Z^{d})}.
\]
\end{lemma}
\begin{proof}
With the notation from Lemma~\ref{lem:var-all-q} we have
\[
\Lf[q,Q,t,\mathrm{poly}] (\vec\alpha)
=
\Lf[q,Q,t] (\vec\alpha) \big( \sum_{a_{1},\dots,a_{d}=1}^{q} S(a/q) \cutofff[Q/4](\vec\alpha-\frac aq) \big).
\]
In view of that lemma it suffices to show that the inverse Fourier transform of the function in the brackets is uniformly bounded in $\ell^{1}(\Z^{d})$.
The value of that inverse Fourier transform at $x\in\Z^{d}$ equals
\begin{align*}
&\int_{(\R/\Z)^{d}} \sum_{a_{1},\dots,a_{d}=1}^{q} S(a/q) \cutofff[Q/4](\vec\alpha-\frac aq) e(\alpha_{1}x_{1}+\dots+\alpha_{1}x_{d}) \dif\alpha_{1}\dots\dif\alpha_{d}\\
&=
\sum_{a_{1},\dots,a_{d}=1}^{q} S(a/q) \int_{(\R/\Z)^{d}} \cutofff[Q/4](\vec\alpha) e(\alpha_{1}x_{1}+\dots+\alpha_{1}x_{d}) e(\frac{a_{1}}{q}x_{1}+\dots+\frac{a_{d}}{q}x_{d}) \dif\alpha_{1}\dots\dif\alpha_{d}\\
&=
\cutoff[Q/4](x) \sum_{a_{1},\dots,a_{d}=1}^{q} S(a/q) e(\frac{a_{1}}{q}x_{1}+\dots+\frac{a_{d}}{q}x_{d})\\
&=
\cutoff[Q/4](x) \sum_{a_{1},\dots,a_{d}=1}^{q} \frac1q \sum_{n=1}^{q} e(-\frac{a_{1}}{q}n^{1}-\dots-\frac{a_{d}}{q}n^{d}) e(\frac{a_{1}}{q}x_{1}+\dots+\frac{a_{d}}{q}x_{d})\\
&=
\cutoff[Q/4](x) \frac1q \sum_{n=1}^{q} \prod_{j=1}^{d} \sum_{a_{j}=1}^{q} e(\frac{a_{j}}{q}(x_{j}-n^{j}))\\
&=
\cutoff[Q/4](x) q^{d-1} \sum_{n=1}^{q} \prod_{j=1}^{d} 1_{x_{j} \equiv n^{j} \mod q}\\
&=
\cutoff[Q/4](x) q^{d-1} 1_{x_{j} \equiv x_{1}^{j} \mod q, j\geq 2}.
\end{align*}
This is uniformly bounded in $\ell^{1}$ for $q\lesssim Q$.
To see this note that the characteristic function selects exactly $q$ points from every cube with side length $q$, so we have
\begin{align*}
\| \cutoff[Q/4](x) q^{d-1} 1_{x_{j} \equiv x_{1}^{j} \mod q, j\geq 2} \|_{\ell^{1}}
&\leq
q^{d} \sum_{y\in\Z^{d}} \sup_{x\in qy + [1,q]^{d}} \eta_{Q/4}(x)\\
&\lesssim
q^{d} Q^{-d} \sum_{y\in\Z^{d}} \sup_{x\in qy + [1,q]^{d}} \eta(4x/Q)\\
&\lesssim
q^{d} Q^{-d} \sum_{y\in\Z^{d}} \sup_{x\in qy + [1,q]^{d}} (1+|x|/Q)^{-d-1}\\
&\lesssim
q^{d} Q^{-d} \sum_{y\in\Z^{d}, |y|<CQ/q} 1
+
q^{d} Q^{-d} \sum_{y\in\Z^{d}, |y|\geq CQ/q} (q|y|/Q)^{-d-1}\\
&\lesssim
1.
\qedhere
\end{align*}
\end{proof}

\begin{proof}[Proof of Theorem~\ref{thm:var-poly} for $\big| \frac1p - \frac12 \big| < \frac{1}{2(d+1)}$]
Repeating the argument leading to Corollary~\ref{cor:var-rats} we obtain
\[
\| \| \Ls[s,t] * f \|_{\iV^{r}_{t>0}} \|_{\ell^{p}(\Z^{d})} \lesssim_{p,r,\epsilon} 2^{s(1+\epsilon)} \| f \|_{\ell^{p}(\Z^{d})}
\]
for $1<p<\infty$.
Recall that in view of Proposition~\ref{prop:multiple-freq-lp-circle} we can replace $(1+\epsilon)$ by $(-1/d+\epsilon)$ for $p=2$.
Interpolation gives
\[
\| \| \Ls[s,t] * f \|_{\iV^{r}_{t>0}} \|_{\ell^{p}(\Z^{d})} \lesssim_{p,r,\epsilon} 2^{s(-(1-|1-2/p|)/d+|1-2/p|+\epsilon)}  \| f \|_{\ell^{p}(\Z^{d})},
\]
and this is summable provided
\[
|1/2-1/p| < 1/2(d+1).
\]
The remaining ingredients of the proof apply unchanged.
\end{proof}

\printbibliography
\end{document}
